\documentclass[a4paper]{amsart}
\usepackage{graphicx}
\usepackage{amssymb}
\usepackage{amsmath}
\usepackage{amsthm}
\usepackage{amscd}
\usepackage[all,2cell]{xy}
\UseAllTwocells \SilentMatrices
\newtheorem{thm}{Theorem}[section]

\newtheorem{cor}[thm]{Corollary}
\newtheorem{lem}[thm]{Lemma}
\newtheorem{exm}[thm]{Example}

\newtheorem{prop}[thm]{Proposition}
\theoremstyle{definition}
\newtheorem{defn}[thm]{Definition}
\theoremstyle{remark}
\newtheorem{rem}[thm]{\bf Remark}
\numberwithin{equation}{section}

\begin{document}
\title[Liftable derived equivalences and objective categories]{Liftable derived equivalences and objective categories}
\author[Xiaofa Chen, Xiao-Wu Chen] {Xiaofa Chen, Xiao-Wu Chen$^*$}

\thanks{$^*$ The corresponding author}
%\thanks{}
\date{\today}

\thanks{E-mail: cxf2011$\symbol{64}$mail.ustc.edu.cn, xwchen$\symbol{64}$mail.ustc.edu.cn}

\maketitle

\dedicatory{}%
\commby{}%
%\begin{center}
%\end{center}

\begin{abstract}
We give two proofs to  the following theorem and its generalization: if a finite dimensional algebra $A$ is derived equivalent to a smooth projective scheme, then any derived equivalence between $A$ and another algebra $B$ is standard, that is, isomorphic to the derived tensor functor by a two-sided tilting complex. The main ingredients of the proofs are as follows: (1)  between the derived categories of two module categories, liftable functors coincide with standard functors; (2) any derived equivalence between a module category and an abelian category is uniquely factorized as the composition of a pseudo-identity and a liftable derived equivalence; (3) the derived category of coherent sheaves on a certain projective scheme is triangle-objective, that is, any triangle autoequivalence on it, which  preserves the the isomorphism classes of complexes, is necessarily  isomorphic to the identity functor.
\end{abstract}

\section{Introduction}

Let $k$ be a field. For a finite dimensional $k$-algebra $A$, we denote by $A\mbox{-mod}$ the abelian category of finitely generated $A$-modules and by $\mathbf{D}^b(A\mbox{-mod})$ its bounded derived category. By a \emph{derived equivalence} between two algebras $A$ and $B$, we mean a $k$-linear triangle equivalence $F\colon \mathbf{D}^b(A\mbox{-mod})\rightarrow \mathbf{D}^b(B\mbox{-mod})$. It is a well-known open question \cite{Ric} whether any derived equivalence is \emph{standard}, that is, isomorphic to the derived tensor functor by a two-sided tilting complex. We refer to the introduction of \cite{CZ} for known cases where the question is answered affirmatively.

The geometric analogue of standard functors are Fourier-Mukai functors, where two-sided tilting complexes are replaced by Fourier-Mukai kernels. The famous theorem in \cite{Or} states that any derived equivalence between smooth projective schemes is a Fourier-Mukai functor.

We are inspired by the following theorem, which seems to be folklore. It provides a large class of algebras, for which the above open question is answered affirmatively.

\vskip 3pt

\noindent {\bf Theorem.}\quad \emph{Let $A$ be a finite dimensional algebra. Assume that there is a derived equivalence between $A$ and a smooth projective scheme. Then any derived equivalence   $F\colon \mathbf{D}^b(A\mbox{-}{\rm mod})\rightarrow \mathbf{D}^b(B\mbox{-}{\rm mod})$ is standard.}

\vskip 3pt

The goal is to give a proof to this theorem and its generalization. Indeed, we give two proofs. The first proof uses the homotopy category of small dg categories and dg lifts of triangle functors, while  the second one is more elementary and uses the notion of triangle-objective categories.

Let us describe the content of this paper. In Section 2, we recall basic facts about dg categories and dg enhancements. In Section 3, we recall the homotopy category of small dg categories and the notion of liftable functors. In Section 4, we prove that between the bounded derived categories of two module categories,  liftable functors coincide with standard functors; see Theorem \ref{thm:lift}. We mention that this result also seems to be folklore.

In Section 5, we prove the following factorization theorem: any derived equivalence between a module category and an abelian category is uniquely factorized as the composition of a pseudo-identity  in the sense of \cite{CY} and a liftable derived equivalence; see Theorem \ref{thm:fac}. Then we give the first proof to the above theorem.

In Section 6, we introduce the following notion of triangle-objective categories: a triangulated category is \emph{triangle-objective},  if any  triangle autoequivalence on it, which preserves the isomorphism classes of objects, is isomorphic to the identity functor. We prove that the bounded derived category of coherent sheaves on a certain projective scheme is triangle-objective; see Proposition~\ref{prop:coh}. It implies the above theorem, when the field $k$ is algebraically closed.

Throughout, we work over a fixed field $k$. All algebras, categories and functors are required to be $k$-linear. The word dg stands for ``differential graded". In the dg setting, all morphisms and elements are by default homogeneous. Modules are by default left modules.

\section{DG categories and enhancements}

In this section, we recall basic facts and notation for dg categories and enhancements. The standard references for dg categories are \cite{Kel94, Dri}.

Let $\mathcal{C}$ be a dg category. For two objects $X$ and $Y$, the Hom complex is denoted by $\mathcal{C}(X, Y)=(\bigoplus_{p\in \mathbb{Z}} \mathcal{C}(X, Y)^p, d=d_{X, Y})$, where $d$ is the differential of degree one satisfying the graded Leibniz rule.  An element in the  subspace $\mathcal{C}(X, Y)^p$ will be called a homogeneous  morphism of degree $p$ with the notation $|f|=p$.

We denote by $H^0(\mathcal{C})$ the \emph{homotopy category} of $\mathcal{C}$, which has the same objects as $\mathcal{C}$ and whose Hom spaces are given by the zeroth cohomologies $H^0(\mathcal{C}(X, Y))$. Similarly, one has the category $Z^0(\mathcal{C})$, whose Hom spaces are given by the zeroth cocycles $Z^0(\mathcal{C}(X, Y))$.

The \emph{opposite} dg category $\mathcal{C}^{\rm op}$ has the same objects and Hom complexes as  $\mathcal{C}$, whose composition $f'\circ^{\rm op} f$ of morphisms $f'$ and $f$ is given by $(-1)^{|f|\cdot |f'|} f\circ f'$. For two dg categories $\mathcal{C}$ and $\mathcal{D}$, we have their  tensor dg category $\mathcal{C}\otimes \mathcal{D}$, whose objects are  the pairs $(C, D)$  with $C\in \mathcal{C}$ and $D\in \mathcal{D}$, and whose Hom complexes are the tensor product of the corresponding Hom complexes in $\mathcal{C}$ and $\mathcal{D}$.

In the following examples, we fix the notation for our concerned dg categories. Let $A$ be a finite dimensional algebra. Denote by $A\mbox{-Mod}$ the category of left $A$-modules. In particular, $k\mbox{-Mod}$ denotes the category of $k$-vector spaces.

\begin{exm}\label{exm:C_dg}
{\rm  Let $\mathcal{A}$ be an additive category. A complex in $\mathcal{A}$ is denoted by $X=(\bigoplus_{p\in \mathbb{Z}} X^p, d_X)$, where  the differentials $d_X^p\colon X^p\rightarrow X^{p+1}$ satisfy $d_X^{p+1}\circ d_X^p=0$. We denote by $C_{\rm dg}(\mathcal{A})$ the dg category formed by  complexes in $\mathcal{A}$. The $p$-th component  of the Hom complex $C_{\rm dg}(\mathcal{A})(X, Y)$ is given by
    $$C_{\rm dg}(\mathcal{A})(X, Y)^p=\prod_{n\in \mathbb{Z}} {\rm Hom}_\mathcal{A}(X^n, Y^{n+p}),$$
    whose elements will be denoted by $f=\{f^n\}_{n\in \mathbb{Z}}$. The differential $d$ acts on $f$ such that $d(f)^n=d_Y^{n+p}\circ f^n-(-1)^pf^{n+1}\circ d_X^n$ for each $n\in \mathbb{Z}$. We are also interested in the full dg subcategory $C^b_{\rm dg}(\mathcal{A})$ formed by bounded complexes.

    We observe that its homotopy category $H^0(C_{\rm dg}(\mathcal{A}))$ coincides with the classical  homotopy category $\mathbf{K}(\mathcal{A})$ of complexes in $\mathcal{A}$, where $H^0(C^b_{\rm dg}(\mathcal{A}))$ corresponds to the bounded homotopy category $\mathbf{K}^b(\mathcal{A})$.

      For two complexes $X$ and $Y$ of $A$-modules, the traditional notation of the Hom complex $C_{\rm dg}(A\mbox{-Mod})(X, Y)$ is ${\rm Hom}_A(X, Y)$.}
    \end{exm}

\begin{exm} {\rm
The dg category $C_{\rm dg}(k\mbox{-Mod})$ is usually denoted by $C_{\rm dg}(k)$. Let $\mathcal{C}$ be a dg category. By a left dg $\mathcal{C}$-module, we mean a dg functor $M\colon \mathcal{C}\rightarrow C_{\rm dg}(k)$. The following notation will be convenient: for a morphism $f\colon X\rightarrow Y$ in $\mathcal{C}$ and $m\in M(X)$, the resulting element $M(f)(m)\in M(Y)$ is written as $f.m$. Here the dot indicates the left $\mathcal{C}$-action on $M$. We denote by $\mathcal{C}\mbox{-{\rm DGMod}}$ the dg category formed by left dg $\mathcal{C}$-modules, whose Hom complexes are defined similarly as in Example~\ref{exm:C_dg}.

Denote by $\mathcal{C}\mbox{-{\rm DGProj}}$ the full dg subcategory of $\mathcal{C}\mbox{-{\rm DGMod}}$ formed by dg-projective $\mathcal{C}$-modules. Here, we recall that a dg $\mathcal{C}$-module is \emph{dg-projective} if and only if it is isomorphic to a direct summand of a semi-free dg $\mathcal{C}$-module in $Z^0(\mathcal{C}\mbox{-{\rm DGMod}})$; compare \cite[3.1]{Kel94} and \cite[Appendix B.1]{Dri}.

    We identify a left $\mathcal{C}^{\rm op}$-modules with a right dg $\mathcal{C}$-module. Then we obtain the dg category $\mbox{{\rm DGMod}-}\mathcal{C}$ of right dg $\mathcal{C}$-modules. For a right dg $\mathcal{C}$-module $N$, a morphism $f\colon X\rightarrow Y$ in $\mathcal{C}$  and $m\in N(Y)$,  the right $\mathcal{C}$-action on $N$ is given such that $m.f=(-1)^{|f|\cdot |m|}N(f)(m)\in N(X)$.

    We identify a dg $\mathcal{C}$-$\mathcal{D}$-bimodule $M\colon  \mathcal{D}^{\rm op} \otimes \mathcal{C} \rightarrow C_{\rm dg}(k)$ with a left dg $\mathcal{C}\otimes \mathcal{D}^{\rm op}$-module. We observe that for each object $C\in \mathcal{C}$, $M(-, C)$ is a right dg $\mathcal{D}$-module. Given a dg functor $F\colon \mathcal{C}\rightarrow \mathcal{D}$, we have a dg $\mathcal{C}$-$\mathcal{D}$-bimodule $M_F$ defined such that $M_F(D, C)=\mathcal{D}(D, F(C))$. }
\end{exm}

\begin{exm}\label{exm:dgpr}
{\rm Let $\mathcal{C}$ be a dg category. Denote by $\mathbf{B}$ the bar resolution of the $\mathcal{C}$-$\mathcal{C}$-bimodule $\mathcal{C}$; see \cite[6.6]{Kel94}. Then we have the following dg functor
$$\mathbf{p}_\mathcal{C}=\mathbf{B}\otimes_\mathcal{C}-\colon \mathcal{C}\mbox{-{\rm DGMod}} \longrightarrow \mathcal{C}\mbox{-{\rm  DGProj}}.$$
For each left dg $\mathcal{C}$-module $M$, $\mathbf{p}_\mathcal{C}(M)$ is a semi-free $\mathcal{C}$-modules, and there is a canonical surjective quasi-isomorphism $\mathbf{p}_\mathcal{C}(M)\rightarrow M$. We call $\mathbf{p}_\mathcal{C}$ the \emph{dg-projective resolution functor} of $\mathcal{C}$.}
\end{exm}

Let $\mathcal{C}$ be a dg category. Recall that both $H^0(\mathcal{C}\mbox{-}{\rm DGMod})$  and $H^0(\mathcal{C}\mbox{-}{\rm DGProj})$  have  natural triangulated structures. The \emph{derived category} $\mathbf{D}(\mathcal{C})$ is the Verdier quotient of $H^0(\mathcal{C}\mbox{-}{\rm DGMod})$ by the triangulated subcategory of  acyclic modules. It is well known that the canonical functor $H^0(\mathcal{C}\mbox{-}{\rm DGProj}) \rightarrow \mathbf{D}(\mathcal{C})$ is a triangle equivalence; see \cite[Theorem 3.1]{Kel94}.

The Yoneda functor
$$\mathbf{Y}_\mathcal{C}\colon \mathcal{C}\longrightarrow {\rm DGMod}\mbox{-}\mathcal{C}, \quad X\mapsto \mathcal{C}(-, X)$$
is a fully-faithful dg functor. In particular, it induces a full embedding
$$H^0(\mathbf{Y}_\mathcal{C})\colon H^0(\mathcal{C}) \longrightarrow H^0({\rm DGMod}\mbox{-}\mathcal{C}).$$
Recall that $H^0({\rm DGMod}\mbox{-}\mathcal{C})$ has a natural triangulated structure. The dg category $\mathcal{C}$ is said to be \emph{pretriangulated}, provided that the essential image of $H^0(\mathbf{Y}_{\mathcal{C}})$ is a triangulated subcategory. The terminology is justified by the evident fact: the homotopy category $H^0(\mathcal{C})$ of a pretriangulated dg category $\mathcal{C}$ has a canonical triangulated structure.

Let $\mathcal{T}$ be a triangulated category. By an \emph{enhancement} of $\mathcal{T}$, we mean a triangle equivalence $E\colon \mathcal{T}\rightarrow H^0(\mathcal{C})$ with $\mathcal{C}$ a pretriangulated dg category. In general, the enhancement is not necessarily unique. We refer to \cite{LO, CS} for more details.

 Let $\mathcal{A}$ be an abelian category. We are mainly concerned with the bounded derived category $\mathbf{D}^b(\mathcal{A})$. As we have seen in Example \ref{exm:C_dg}, $\mathcal{C}^b_{\rm dg}(\mathcal{A})$ provides a canonical enhancement for  $\mathbf{K}^b(\mathcal{A})$. Following \cite[9.8]{Kel05}, we now recall the canonical enhancement of  $\mathbf{D}^b(\mathcal{A})$.

\begin{exm}\label{exm:dgd}
{\rm  Consider the dg category $C^b_{\rm dg}(\mathcal{A})$ of bounded complexes, and its full dg subcategory $C^{b, {\rm ac}}_{\rm dg}(\mathcal{A})$ formed by acyclic complexes. The \emph{ bounded dg derived category} of $\mathcal{A}$ is defined to be  the Drinfeld dg quotient
$$\mathbf{D}_{\rm dg}^b(\mathcal{A})=C^b_{\rm dg}(\mathcal{A})/{C^{b, {\rm ac}}_{\rm dg}(\mathcal{A})}.$$
Recall that the dg category $\mathbf{D}_{\rm dg}^b(\mathcal{A})$ is obtained from $C^b_{\rm dg}(\mathcal{A})$ by freely adding new morphisms $\varepsilon_{X}\colon X\rightarrow X$ of degree $-1$ for each acyclic complex $X$,  such that $d(\varepsilon_X)=1_X$; see \cite[3.1]{Dri}. By \cite[Lemma 1.5]{LO}, $\mathbf{D}_{\rm dg}^b(\mathcal{A})$ is pretriangulated. By \cite[Theorem 3.4]{Dri}, there is a canonical isomorphism of triangulated categories
$${\rm can}_\mathcal{A}\colon \mathbf{D}^b(\mathcal{A})\longrightarrow H^0(\mathbf{D}_{\rm dg}^b(\mathcal{A})),$$
which acts on objects by the identity. We will call ${\rm can}_\mathcal{A}$ the \emph{canonical enhancement} of $\mathbf{D}^b(\mathcal{A})$.}
\end{exm}

\section{The homotopy category and liftable functors}

In this section, we recall the notion of liftable triangle functors between bounded derived categories, and the homotopy category of small dg categories.

Recall that a dg functor $F\colon \mathcal{C}\rightarrow \mathcal{D}$ is a \emph{quasi-equivalence}, provided that the induced chain maps $\mathcal{C}(C, C')\rightarrow \mathcal{D}(F(C), F(C'))$ are all quasi-isomorphisms, and that $H^0(F)\colon H^0(\mathcal{C})\rightarrow H^0(\mathcal{D})$ is dense. In this situation, $H^0(F)$ is an equivalence.

\begin{lem}\label{lem:H0}
Let $F\colon \mathcal{C}\rightarrow \mathcal{D}$ be a dg functor between two pretriangulated dg categories. Assume that $H^0(F)$ is an equivalence. Then $F$ is a quasi-equivalence.
\end{lem}

\begin{proof}
It suffices to show that the induced chain map  $\mathcal{C}(C, C')\rightarrow \mathcal{D}(F(C), F(C'))$ is a quasi-isomorphism. Recall that $H^i(\mathcal{C}(C, C'))$ is isomorphic to ${\rm Hom}_{H^0(\mathcal{C})}(C, \Sigma^i(C'))$, where $\Sigma$ denotes the translation functor on the triangulated category $H^0(\mathcal{C})$. Similarly, we identify $H^i(\mathcal{D}(F(C), F(C')))$ with ${\rm Hom}_{H^0(\mathcal{D})}(F(C), \Sigma^i(F(C')))$. By assumption,  $H^0(F)$ is a triangle  equivalence between triangulated categories $H^0(\mathcal{C})$ and $H^0(\mathcal{D})$; compare \cite[Remark~1.8(i)]{CS}.  Then $H^0(F)$ induces an isomorphism
$${\rm Hom}_{H^0(\mathcal{C})}(C, \Sigma^i(C'))\longrightarrow {\rm Hom}_{H^0(\mathcal{D})}(F(C), \Sigma^i(F(C'))).$$
We infer the required quasi-isomorphism.
\end{proof}

In the following examples, we fix the notation for some quasi-equivalences, which will be used in the next section.

\begin{exm}\label{exm:YD}
 {\rm Recall that a right dg $\mathcal{D}$-module $M$ is \emph{quasi-representable}, provided that it is isomorphic to $\mathcal{D}(-, D)$ in $\mathcal{D}(\mathcal{D}^{\rm op})$ for some object $D$ in $\mathcal{D}$. Denote by $\bar{\mathcal{D}}$ the full dg subcategory of ${\rm DGProj}\mbox{-}\mathcal{D}$ formed by dg-projective quasi-representable modules. Then the Yoneda embedding induces a quasi-equivalence $\mathbf{Y}_\mathcal{D}\colon \mathcal{D}\rightarrow \bar{\mathcal{D}}$. }
\end{exm}

We identify a dg algebra $B$ with a dg category with one object. We denote by $B\mbox{-DGMod}^{\rm fd}$ the full dg subcategory of $B\mbox{-DGMod}$ consisting of those left dg $B$-modules with finite dimensional total cohomologies. Similarly, we have the dg category $B\mbox{-DGProj}^{\rm fd}$.

The following example is implicitly contained in \cite[7.2]{Kel94}.

\begin{exm}\label{exm:quasi}
{\rm Let $\theta \colon C\rightarrow B$ be a quasi-isomorphism between dg algebras. Then there is a quasi-equivalence
$$B\otimes_C-\colon C\mbox{-{\rm DGProj}} \longrightarrow B\mbox{-{\rm DGProj}}.$$
Using infinite devissage, one infers that the natural map $P\rightarrow B\otimes_C P$ is a quasi-isomorphism for any dg-projective $C$-module $P$. In particular, the above quasi-equivalence restricts to a quasi-equivalence
$$B\otimes_C-\colon C\mbox{-{\rm DGProj}}^{\rm fd} \longrightarrow B\mbox{-{\rm DGProj}}^{\rm fd}.$$
}
\end{exm}

We identify a usual algebra  as a dg algebra  concentrated in degree zero. Then dg modules are just complexes of usual modules.  For a finite dimensional algebra $A$, we denote by $A\mbox{-mod}$ the abelian category of finite dimensional left $A$-modules, and by $A\mbox{-proj}$ its full subcategory formed by finitely generated projective modules.

\begin{exm}\label{exm:res}
{\rm Let $A$ be a finite dimensional algebra. Then $A\mbox{-{\rm DGMod}}$ is identified with $C_{\rm dg}(A\mbox{-Mod})$. The dg-projective resolution functor ${\bf p}_A\colon C_{\rm dg}(A\mbox{-Mod})\rightarrow A\mbox{-{\rm DGProj}}$ restricts to
$${\bf p}_A\colon C^b_{\rm dg}(A\mbox{-mod})\rightarrow A\mbox{-{\rm DGProj}}^{\rm fd}.$$
 Since $\mathbf{p}_A$ sends each acyclic complex $X$ to a contractible complex $\mathbf{p}_A(X)$, it induces a dg functor
$$\mathbf{p}'_A\colon \mathbf{D}^b_{\rm dg}(A\mbox{-mod})\longrightarrow A\mbox{-{\rm DGProj}}^{\rm fd}.$$
For the construction, we put $\mathbf{p}'_A(\varepsilon_X)$ to be any contracting homotopy on $\mathbf{p}_A(X)$, where $\varepsilon_X$ is the new generator in defining $\mathbf{D}^b_{\rm dg}(A\mbox{-mod})$; see Example \ref{exm:dgd}. We observe that $\mathbf{p}'_A$ is a quasi-equivalence.  Indeed, taking $H^0(\mathbf{p}'_A)$, we obtain the well-known triangle equivalence $\mathbf{D}^b(A\mbox{-mod})\simeq H^0(A\mbox{-{\rm DGProj}}^{\rm fd})$.

Denote by $C_{\rm dg}^{-, b}(A\mbox{-proj})$ the dg category formed by bounded-above complexes of finitely generated projective $A$-modules, which have bounded cohomologies. Since bounded-above complexes of projective modules are dg-projective, we have the inclusion
$${\rm inc}_A\colon \colon C_{\rm dg}^{-, b}(A\mbox{-proj}) \longrightarrow A\mbox{-{\rm DGProj}}^{\rm fd}.$$
This is a quasi-equivalence. We just recall that there is a well-know triangle equivalence between the homotopy category $\mathbf{K}^{-, b}(A\mbox{-proj})$ and  $\mathbf{D}^b(A\mbox{-mod})$ .
}\end{exm}

We denote by $\mathbf{dgcat}$ the category of small dg categories, whose morphisms are dg functors.  The \emph{homotopy category} $\mathbf{Hodgcat}$ is the localization of $\mathbf{dgcat}$ with respect to all the quasi-equivalences. In other words, $\mathbf{Hodgcat}$ is obtained from $\mathbf{dgcat}$ by formally inverting quasi-equivalences. For two dg categories $\mathcal{C}$ and $\mathcal{D}$, we denote by $[\mathcal{C}, \mathcal{D}]$ the corresponding  Hom set in $\mathbf{Hodgcat}$, whose elements are usually denoted by  $\mathcal{C}\dashrightarrow \mathcal{D}$. We mention that any such morphism can be realised as a roof $\mathcal{C} \xleftarrow{F} \mathcal{C}'\xrightarrow{F'} \mathcal{D}$ of dg functors, where $F$ is a quasi-equivalence; moreover, $F$ can be taken as a semi-free resolution of $\mathcal{C}$; see \cite[Appendix B.5]{Dri}. For details, we refer to \cite{Toe2}.

For the set-theoretical consideration relevant to us, we use the following remark.

\begin{rem}\label{rem:quasi-sm}
We call a dg category $\mathcal{C}$  \emph{quasi-small}, provided that the homotopy category $H^0(\mathcal{C})$ is essentially small. We choose for each isomorphism class in $H^0(\mathcal{C})$ a representative in $\mathcal{C}$. These objects form a small dg full subcategory $\mathcal{C}'$. By the construction, the inclusion $\mathcal{C}'\hookrightarrow \mathcal{C}$ is a quasi-equivalence. So, we identify $\mathcal{C}$ with $\mathcal{C}'$, and view $\mathcal{C}$ as an object in $\mathbf{Hodgcat}$.
\end{rem}

Denote by $[\mathbf{cat}]$ the category of small categories, whose morphisms are the isomorphism classes of functors. In particular, equivalences of categories are isomorphisms in  $[\mathbf{cat}]$. Therefore, the homotopy functor $H^0\colon \mathbf{dgcat}\rightarrow [\mathbf{cat}]$ inverts quasi-equivalences. By the universal property of the localization, we have the induced functor
$$H^0\colon \mathbf{Hodgcat} \longrightarrow [\mathbf{cat}], \quad \mathcal{C}\mapsto H^0(\mathcal{C}). $$

 Following \cite[7.1]{Kel94},  a \emph{quasi-functor} from $\mathcal{C}$ to $\mathcal{D}$ is a dg $\mathcal{C}$-$\mathcal{D}$-bimodule $X$ such that for each object $C\in \mathcal{C}$, the right $\mathcal{D}$-module $X(-, C)$ is quasi-representable.
 We denote by ${\rm rep}(\mathcal{C}, \mathcal{D})$ the full subcategory of  $\mathbf{D}(\mathcal{C}\otimes \mathcal{D}^{\rm op})$ formed by quasi-functors; it is a triangulated subcategory; see \cite[Appendix E.2]{Dri}. We denote by ${\rm Iso}({\rm rep}(\mathcal{C}, \mathcal{D}))$ the set of isomorphism classes of quasi-functors.

 For each quasi-functor $M$, we take its dg-projective resolution  $\mathbf{p}(M)$. The quasi-functor $\mathbf{p}(M)$ defines a dg functor $\mathbf{p}M\colon \mathcal{C}\rightarrow \bar{\mathcal{D}}$ sending $C$ to $(\mathbf{p}M)(-, C)$. The following diagram
$$\mathcal{C}\xrightarrow{\mathbf{p}M} \bar{\mathcal{D}}\xleftarrow{\mathbf{Y}_\mathcal{D}}  \mathcal{D}$$
defines a morphism $\Phi_M\colon \mathcal{C}\dashrightarrow \mathcal{D}$ in $\mathbf{Hodgcat}$. Here, we recall the quasi-equivalence $\mathbf{Y}_\mathcal{D}$ in Example \ref{exm:YD}.

The following bijection is fundamental; see \cite[Sublemmas 4.4 and 4.5]{Toe1} and \cite[p.279, Corollary 1]{Toe2}. For an elementary proof, we refer to \cite{CS15}. As mentioned in \cite[Remark~6.6]{CS}, the morphism $\Phi_M$ might be viewed as the generalized Fourier-Mukai transform with $M$ being its kernel.

\begin{lem}\label{lem:fund}
Keep the notation as above. Then there is a bijection
$${\rm Iso}({\rm rep}(\mathcal{C}, \mathcal{D})) \longrightarrow [\mathcal{C}, \mathcal{D}], \quad M \mapsto \Phi_M.$$
\end{lem}

The following notion is modified  from \cite[Definition 6.7]{CS}. Here, since the uniqueness of the enhancement is not known, we have to fix the canonical one.

\begin{defn}
Let $F\colon \mathbf{D}^b(\mathcal{A}) \rightarrow \mathbf{D}^b(\mathcal{B})$ be a triangle functor. We say that $F$ is \emph{liftable}, provided that there is a morphism $\tilde{F}\colon \mathbf{D}_{\rm dg}^b(\mathcal{A})\dashrightarrow \mathbf{D}_{\rm dg}^b(\mathcal{B})$ in $\mathbf{Hodgcat}$, called a \emph{dg lift} of $F$,  such that $F$ is isomorphic to ${\rm can}_\mathcal{B}^{-1}H^0(\tilde{F}) {\rm can}_\mathcal{A}$.
\end{defn}

We observe that the composition of liftable functors is still liftable. Using the following well-known lemma, we infer that a quasi-inverse of a liftable equivalence is also liftable. We point out that liftable functors are called standard in \cite[9.8]{Kel05}. However, we reserve the terminology ``standard functors" for the classical ones, that is, derived tensor functors by complexes.

\begin{lem}
Let $F\colon \mathbf{D}^b(\mathcal{A}) \rightarrow \mathbf{D}^b(\mathcal{B})$ be a triangle equivalence. Then any dg lift  $\tilde{F}$ of $F$ is an isomorphism in $\mathbf{Hodgcat}$.
\end{lem}

\begin{proof}
We use the roof presentation  $\mathbf{D}_{\rm dg}^b(\mathcal{A})\xleftarrow{F_1} \mathcal{C}\xrightarrow{F_2} \mathbf{D}_{\rm dg}^b(\mathcal{B})$ of $\tilde{F}$, where $F_1$ is a quasi-equivalence. It follows that the dg category $\mathcal{C}$  is also pretriangulated. By assumption, we infer that $H^0(F_2)$ is an equivalence. By Lemma \ref{lem:H0}, the dg functor $F_2$ is a quasi-equivalence, which implies that $\tilde{F}$ is an isomorphism.
\end{proof}

\section{Liftable and standard functors}

In this section, we prove that the  category of quasi-functors between the bounded dg derived categories of  two module categories is triangle equivalent to a certain derived category of bimodules over the given algebras. Consequently, between the bounded derived categories of two module categories, liftable functors coincide with standard functors.

Let $A$ and $B$ be two finite dimensional algebras. Recall from \cite[Definition~3.4]{Ric} that a triangle functor $F\colon \mathbf{D}^b(A\mbox{-}{\rm mod}) \rightarrow \mathbf{D}^b(B\mbox{-}{\rm mod})$ is \emph{standard}, provided that $F\simeq X\otimes^\mathbb{L}_A-$ for some bounded complex $X$ of $B$-$A$-bimodules. Since $X\otimes^\mathbb{L}_A-$ sends bounded complexes to bounded complexes, we infer that the underlying complex $X_A$ of right $A$-modules is \emph{perfect}, that is, isomorphic to some object in $\mathbf{K}^b(A^{\rm op}\mbox{-proj})$.

We will identify $\mathbf{D}^b(A\mbox{-}{\rm mod})$ with $\mathbf{K}^{-, b}(A\mbox{-proj})$; compare Example \ref{exm:res}. For each complex $P\in \mathbf{K}^{-, b}(A\mbox{-}{\rm proj})$ and $N\geq 0$, we consider the brutal truncation $\sigma_{\geq -N}(P)$, which is a subcomplex of $P$ consisting of $P^n$ for $n\geq -N$. The inclusion ${\rm inc}_N\colon \sigma_{\geq -N}(P)\rightarrow P$ fits into a canonical exact triangle
\begin{align}\label{equ:tri1}
\sigma_{\geq -N}(P) \xrightarrow{{\rm inc}_N}  P \longrightarrow  \sigma_{<-N}(P) \longrightarrow \Sigma \sigma_{\geq -N}(P),
\end{align}
where $\sigma_{<-N} P=P/{\sigma_{\geq -N}(P) }$ is the quotient complex.

The following result is standard.

\begin{lem}\label{lem:H0iso}
Let $F \colon \mathbf{K}^{-, b}(A\mbox{-}{\rm proj})\rightarrow \mathbf{K}^{-, b}(B\mbox{-}{\rm proj})$ be a triangle functor. Then there is a natural number $N_0$ such that $H^0(F({\rm inc}_{N}))$ is an isomorphism for each complex $P$ and $N\geq N_0$.
\end{lem}

\begin{proof}
For each interval $I$, we denote by $\mathcal{D}_A^I$ the full subcategory of $\mathbf{K}^{-, b}(A\mbox{-}{\rm proj})$ formed by those complexes $X$ satisfying $H^i(X)=0$ for $i\notin I$. Similarly, we have the subcategories $\mathcal{D}^I_B$ of $\mathbf{K}^{-, b}(B\mbox{-}{\rm proj})$ . These subcategories are closed under extensions.

The subcategory $\mathcal{D}_A^{[0, 0]}$ is equivalent to $A\mbox{-mod}$, which has only finitely many simple $A$-modules up to isomorphism. It follows that $F(\mathcal{D}_A^{[0, 0]})\subseteq \mathcal{D}_B^{[-N_0+1, N_0-1]}$ for $N_0>0$ large enough. More generally, we have $F(\mathcal{D}_A^{[a, b]})\subseteq \mathcal{D}_B^{[a-N_0+1, b+N_0-1]}$. It follows that $F(\sigma_{<-N}(P)) \in \mathcal{D}_B^{(-\infty, -2]}$ for each $N\geq N_0$. Applying the triangle functor $F$ to  (\ref{equ:tri1}) and taking cohomologies, we infer the required isomorphism.
\end{proof}

We denote by $\mathbf{D}(B\otimes A^{\rm op})$ the derived category of complexes of  $B$-$A$-bimodules.

\begin{thm}\label{thm:lift}
There is a triangle equivalence
$${\rm rep}(\mathbf{D}_{\rm dg}^b(A\mbox{-{\rm mod}}), \mathbf{D}_{\rm dg}^b(B\mbox{-{\rm mod}}))\stackrel{\sim}\longrightarrow \{M\in \mathbf{D}(B\otimes A^{\rm op})\; |\; M_A \mbox{ is perfect}\},$$
sending a dg $\mathbf{D}_{\rm dg}^b(A\mbox{-{\rm mod}})$-$\mathbf{D}_{\rm dg}^b(B\mbox{-{\rm mod}})$-bimodule $X$ to $X(B, A)$.

 Consequently, a triangle functor $F\colon \mathbf{D}^b(A\mbox{-}{\rm mod}) \rightarrow \mathbf{D}^b(B\mbox{-}{\rm mod})$ is liftable if and only if it is standard.
\end{thm}

\begin{proof}
We use the sequence of quasi-equivalences in Example \ref{exm:res}
$$\mathbf{D}^b_{\rm dg}(A\mbox{-mod})\xrightarrow{\mathbf{p}'_A} A\mbox{-{\rm DGProj}}^{\rm fd} \xleftarrow{{\rm inc}_A} C_{\rm dg}^{-, b}(A\mbox{-proj}). $$
In this proof, we identify $\mathbf{D}^b_{\rm dg}(A\mbox{-mod})$ with $\mathcal{C}:=C_{\rm dg}^{-, b}(A\mbox{-proj})$, $\mathbf{D}^b_{\rm dg}(B\mbox{-mod})$ with $\mathcal{D}:=C_{\rm dg}^{-, b}(B\mbox{-proj})$.

We will actually prove that there is a triangle equivalence
$${\rm rep}(\mathcal{C}, \mathcal{D})\stackrel{\sim}\longrightarrow \{M\in \mathbf{D}(B\otimes A^{\rm op})\; |\; M_A \mbox{ is perfect}\},$$
sending $X$ to $X(B, A)$, whose quasi-inverse sends $M$ to the dg $\mathcal{C}$-$\mathcal{D}$-bimodule $X_M$ defined by $X_M(Q, P)={\rm Hom}_B(Q, M\otimes_A P)$ for $P\in \mathcal{C}$ and $Q\in \mathcal{D}$.

Take $X\in {\rm rep}(\mathcal{C}, \mathcal{D})$ and fix an isomorphism
$$\xi_{-, P}\colon X(-,P)\stackrel{\sim}\longrightarrow \mathcal{D}(-, F(P))$$
in $\mathbf{D}(\mathcal{D}^{\rm op})$ for each $P\in \mathcal{C}$. Therefore, the dg bimodule $X$ induces a triangle functor
$$F\colon \mathbf{K}^{-, b}(A\mbox{-proj})\longrightarrow \mathbf{K}^{-, b}(B\mbox{-proj}).$$
 In particular, $X(B, P)$ is isomorphic to $F(P)$ in $\mathbf{D}(B)$, which has bounded cohomologies. We claim that the following natural map
\begin{align}\label{equ:iso1}
\theta_P\colon X(B, A)\otimes_A P\longrightarrow X(B, P), \quad m\otimes p\mapsto (-1)^{|m|\cdot |p|}p.m
\end{align}
is a quasi-isomorphism. Here, we view $p\in P$ as an element in $\mathcal{C}(A, P)$, and then $p.m$ denotes the left $\mathcal{C}$-action on $X$. By the claim, $X(B, A)\otimes_A P$ has bounded cohomologies for each  $P\in \mathcal{C}$. Therefore, the underlying complex  $X(B, A)_A$ of right $A$-modules is perfect.

We observe that $\theta_P$ is an isomorphism in the case that $P\simeq \Sigma^i(A)$. It follows that $\theta_P$ is an isomorphism for any bounded complex $P$ in $\mathcal{C}$. In general, we will show that $H^i(\theta_P)$ is an isomorphism. By translation, we will only show that $H^0(\theta_P)$ is an isomorphism. We consider the brutal truncation $\sigma_{\geq -N}(P)$, which is a bounded subcomplex of $P$. The inclusion ${\rm inc}_N\colon \sigma_{\geq -N}(P)\rightarrow P$ induces the vertical maps in  the following commutative diagram.
\[
\xymatrix{
X(B, A)\otimes_A \sigma_{\geq -N}(P) \ar[d] \ar[rr]^-{\theta_{\sigma_{\geq -N}(P)}} && X(B, \sigma_{\geq -N}(P)) \ar[d] \ar[rr]^{\sim} && F(\sigma_{\geq -N}(P))\ar[d]\\
X(B, A)\otimes_A P \ar[rr]^-{\theta_P} && X(B, P) \ar[rr]^{\sim} && F(P)
}\]
Since $X(B, A)$ has bounded cohomologies, the leftmost vertical map induces an isomorphism on $H^0$ for sufficiently large $N$. By Lemma \ref{lem:H0iso}, a similar remark holds for the rightmost one. Then the claim follows from the isomorphism $\theta_{{\sigma_{\geq -N}(P)}}$.

For each $Q\in \mathcal{D}$,  we claim that the following natural map
\begin{align}\label{equ:iso2}
\delta\colon X(Q, P)\longrightarrow {\rm Hom}_B(Q, X(B, P)), \quad x\mapsto (q\mapsto x.q)
\end{align}
 is a quasi-isomorphism. Here, $q\in Q$ is viewed as an element in $\mathcal{D}(B, Q)$, and $x.q$ denotes the right $\mathcal{D}$-action on $X$.

 To see the claim, we use the isomorphism $\xi_{B, P}\colon X(B, P)\rightarrow F(P)$ in $\mathbf{D}(B)$. Then we have the following quasi-isomorphisms of complexes
   $${\rm Hom}_B(Q, X(B, P))\simeq {\rm Hom}_B(Q, F(P))=\mathcal{D}(Q, F(P))\simeq X(Q, P).$$
 Then the  claim follows, since $\delta$ is compatible with the above quasi-isomorphisms.

 Combining the quasi-isomorphisms (\ref{equ:iso1}) and (\ref{equ:iso2}), we obtain a roof of quasi-isomorphisms
 \begin{align}\label{equ:iso3}
 X(Q, P)\xrightarrow{\delta} {\rm Hom}_B(Q, X(B, P)) \xleftarrow{{\rm Hom}_B(Q, \theta_P)} {\rm Hom}_B(Q, X(B, A)\otimes_A P).
 \end{align}
 This gives rise to an isomorphism of dg $\mathcal{C}$-$\mathcal{D}$-bimodules in $\mathbf{D}(\mathcal{C}\otimes \mathcal{D}^{\rm op})$. Then the required mutually inverse equivalences follows immediately.

 It remains to prove the consequence. The ``if" part is well known. Assume that $F\simeq M\otimes^\mathbb{L}_A-$ for a bounded complex $M$ of $B$-$A$-bimodules with $M_A$  perfect. By using projective resolutions of $B$-$A$-bimodules and truncations, we may assume that $M_A$ lies in $\mathbf{K}^b(A^{\rm op}\mbox{-proj})$. Then the dg functor
 $$M\otimes_A-\colon \mathbf{D}^b_{\rm dg}(A\mbox{-mod})\longrightarrow \mathbf{D}^b_{\rm dg}(B\mbox{-mod})$$
 is well-defined, which is a dg lift of $F$.

 For the ``only if" part, we assume that $F$ admits a dg lift $\tilde{F}\colon \mathcal{C}\dashrightarrow \mathcal{D}$. By Lemma~\ref{lem:fund}, we may assume that $\tilde{F}=\Phi_X$ for some $X\in {\rm rep}(\mathcal{C}, \mathcal{D})$. We identify $X$ with its dg-projective resolution $\mathbf{p}(X)$. By definition, $H^0(\Phi_X)(P)$ is representing the right dg $\mathcal{D}$-module $X(-, P)$. By (\ref{equ:iso3}), we infer that $H^0(\Phi_X)(P)$ is isomorphic to $X(B, A)\otimes_AP$. More precisely, we might replace $X(B, A)$ by a bounded-above complex $M$ of finitely generated projective $B$-$A$-bmodules. Then $H^0(\Phi_X)(P)$  is isomorphic to $M\otimes_A P$. Consequently, $F$ is identified with
 $$M\otimes_A- \colon \mathbf{K}^{-, b}(A\mbox{-proj})\longrightarrow \mathbf{K}^{-, b}(B\mbox{-proj}).$$
 This proves that $F$ is standard.
 \end{proof}

\section{A factorization theorem for derived equivalences}

In this section, we prove a factorization theorem for derived equivalences: any derived equivalence between a module category and an abelian category is a composition of a pseudo-identity  with a liftable derived equivalence.

The following notions are taken from \cite[Definitions 3.8 and 5.1]{CY}; compare \cite[Lemma 5.2]{CY}. For an abelian category $\mathcal{B}$, we identify $\mathcal{B}$ as the full subcategory of $\mathbf{D}^b(\mathcal{B})$ formed by stalk complexes concentrated in degree zero. More generally, we denote by $\Sigma^{n}(\mathcal{B})$ the full subcategory formed by stalk complexes concentrated in degree $-n$.

\begin{defn}
We call a triangle functor $F\colon \mathbf{D}^b(\mathcal{B})\rightarrow \mathbf{D}^b(\mathcal{B})$ a \emph{pseudo-identity}, provided that $F(X)=X$ for each complex $X$, and that for each integer $n$,  the restriction $F|_{\Sigma^{n}(\mathcal{B})}\colon \Sigma^n(\mathcal{B})\rightarrow \Sigma^n(\mathcal{B})$ is the identity functor.

The abelian category $\mathcal{B}$ is called \emph{$\mathbf{D}$-standard}, provided that any pseudo-identity on $\mathbf{D}^b(\mathcal{B})$ is isomorphic, as triangle functors,  to the identity functor.
\end{defn}

We observe that a pseudo-identity is necessarily an autoequivalence; see \cite[Lemma 3.6]{CY}. The main motivation of introducing $\mathbf{D}$-standard categories is the following result: the module category $A\mbox{-mod}$ for a finite dimensional algebra $A$ is $\mathbf{D}$-standard if and only if any derived equivalence $F\colon \mathbf{D}^b(A\mbox{-mod})\rightarrow \mathbf{D}^b(B\mbox{-mod})$ is standard; see \cite[Theorem 5.10]{CY}. Therefore, the well-known open question \cite{Ric} about standard derived equivalences is equivalent to the conjecture that any module category $A\mbox{-mod}$ is $\mathbf{D}$-standard. On the other hand, there exists a triangle functor between the bounded derived categories of module categories, which is neither an equivalence nor standard; see \cite[Corollary 1.5]{RVan}.

In what follows, $A$ will be a finite dimensional algebra and $\mathcal{A}$ an abelian category.

\begin{lem}\label{lem:pseu}
Let $F\colon \mathbf{D}^b(A\mbox{-}{\rm mod})\rightarrow \mathbf{D}^b(A\mbox{-}{\rm mod})$ be a triangle functor. Assume that there is an isomorphism $\theta \colon F(A)\rightarrow A$ such that $\theta \circ F(a)=a\circ \theta$ for each morphism $a\colon A\rightarrow A$. Then $F$ is isomorphic to a pseudo-identity.
\end{lem}

\begin{proof}
We observe that $F$ induces isomorphisms
$${\rm Hom}_{\mathbf{D}^b(A\mbox{-}{\rm mod})}(A, \Sigma^n(A)) \longrightarrow {\rm Hom}_{\mathbf{D}^b(A\mbox{-}{\rm mod})}(F(A), F\Sigma^n(A))$$
for each integer $n$. The cases $n\neq 0$ are trivial, since both sides equal zero. If $n=0$, we just use the assumption $F(a)=\theta^{-1}\circ a\circ \theta$ for each endomorphism $a$ on $A$.

 We identify  $\mathbf{K}^b(A\mbox{-}{\rm proj})$ with the smallest triangulated subcategory of $\mathbf{D}^b(A\mbox{-}{\rm mod})$ containing $A$ and closed under direct summands. We observe that $F(\mathbf{K}^b(A\mbox{-}{\rm proj}))\subseteq \mathbf{K}^b(A\mbox{-}{\rm proj})$. By Beilinson's Lemma, the restriction $F|_{\mathbf{K}^b(A\mbox{-}{\rm proj})}$ is an equivalence. Then $F$ is an autoequivalence by applying the last statement in \cite[Proposition 3.4]{C} or, alternatively by the equivalence in \cite[Theorem 6.2]{Kr}.

 Recall that a complex $X$ lies in $A\mbox{-mod}$ if and only if ${\rm Hom}_{\mathbf{D}^b(A\mbox{-}{\rm mod})}(A, \Sigma^n(X))=0$ for $n\neq 0$. It follows from the equivalence $F$ that $F(X)$ lies in $A\mbox{-mod}$ for $X\in A\mbox{-mod}$. So, we have the restriction $F|_{A\mbox{-}{\rm mod}}\colon A\mbox{-mod}\rightarrow A\mbox{-mod}$. By the isomorphism $\theta$, it is standard  to see that $F|_{A\mbox{-}{\rm mod}}$ is isomorphic to the identity functor. Then we are done by \cite[Corollary 3.9]{CY}.
\end{proof}

 The following factorization theorem extends \cite[Proposition 5.8]{CY}, which is essentially due to \cite[Corollary 3.5]{Ric}.

\begin{thm}\label{thm:fac}
Let $F\colon \mathbf{D}^b(A\mbox{-{\rm mod}})\rightarrow \mathbf{D}^b(\mathcal{A})$ be a triangle equivalence. Then there is a factorization $F\simeq F_2F_1$ of triangle functors, where $F_1\colon \mathbf{D}^b(A\mbox{-{\rm mod}})\rightarrow \mathbf{D}^b(A\mbox{-{\rm mod}})$ is a pseudo-identity and $F_2\colon \mathbf{D}^b(A\mbox{-{\rm mod}})\rightarrow \mathbf{D}^b(\mathcal{A})$ is a liftable equivalence.

Moreover, such a factorization is unique. More precisely, for another factorization $F\simeq F'_2F'_1$ with $F'_1$ a pseudo-identity on  $\mathbf{D}^b(A\mbox{-{\rm mod}})$ and $F'_2$ a liftable equivalence, we have $F_1\simeq F'_1$ and $F_2\simeq F'_2$.
\end{thm}

\begin{proof}
Set $T=F(A)$,  and $\Gamma={\rm End}_{\mathbf{D}_{\rm dg}^b(\mathcal{A})}(T)^{\rm op}$ to be the opposite dg endomorphism algebra of $T$ in $\mathbf{D}_{\rm dg}^b(\mathcal{A})$.

Recall that $H^n(\Gamma)$ is isomorphic to ${\rm Hom}_{\mathbf{D}^b(\mathcal{A})} (T, \Sigma^n(T))$ for each integer $n$. By the equivalence $F$, we infer that $H^n(\Gamma)=0$ for $n\neq 0$ and that $H^0(\Gamma)$ is isomorphic to $A$. Denote by $\tau_{\leq 0}(\Gamma)$ the good truncation of $\Gamma$, that is,  $\tau_{\leq 0}(\Gamma)=\bigoplus_{i<0} \Gamma^i \oplus {\rm Ker} d_\Gamma^0$. Then $\tau_{\leq 0}(\Gamma)$ is a dg subalgebra of $\Gamma$ and $H^0(\Gamma)$ is a quotient algebra of $\tau_{\leq 0}(\Gamma)$. Therefore, we have quasi-isomorphisms of dg algebras $\Gamma  \hookleftarrow\tau_{\leq 0}(\Gamma)\twoheadrightarrow A$.

 For each object $X\in \mathbf{D}_{\rm dg}^b(\mathcal{A})$, $\mathbf{D}_{\rm dg}^b(\mathcal{A})(T, X)$ is naturally a left dg $\Gamma$-module. We observe that $H^n(\mathbf{D}_{\rm dg}^b(\mathcal{A})(T, X))$ is isomorphic to ${\rm Hom}_{\mathbf{D}^b(\mathcal{A})}(T, \Sigma^n(X))$, which is further isomorphic to ${\rm Hom}_{\mathbf{D}^b(A\mbox{-{\rm mod}})}(A, \Sigma^nF^{-1}(X))$ by the equivalence $F$. It follows that $\mathbf{D}_{\rm dg}^b(\mathcal{A})(T, X)$ lies in $\Gamma\mbox{-{\rm DGMod}}^{\rm fd}$.

  We define a morphism $\tilde{F}\colon  \mathbf{D}_{\rm dg}^b(\mathcal{A}) \dashrightarrow  \mathbf{D}_{\rm dg}^b(A\mbox{-mod})$ by the following diagram.
 \[\xymatrix{
 \mathbf{D}_{\rm dg}^b(\mathcal{A}) \ar@{-->}[d]_-
 {\tilde{F}} \ar[rr]^-{\mathbf{D}_{\rm dg}^b(\mathcal{A})(T, -)} && \Gamma\mbox{-{\rm DGMod}}^{\rm fd}  \ar[rr]^-{\mathbf{p}_\Gamma} && \Gamma\mbox{-{\rm DGProj}}^{\rm fd}\\
 \mathbf{D}_{\rm dg}^b(A\mbox{-{\rm mod}}) \ar[rr]^-{\mathbf{p}'_A} && A\mbox{-{\rm DGProj}}^{\rm fd} &&   \tau_{\leq 0}(\Gamma)\mbox{-{\rm DGProj}}^{\rm fd}  \ar[ll]_-{A\otimes_{\tau_{\leq 0}(\Gamma)}-} \ar[u]_-{\Gamma\otimes_{\tau_{\leq 0}(\Gamma)}-}
 }\]
For the dg-projective resolution functor $\mathbf{p}_\Gamma$, we refer to Example \ref{exm:dgpr}. For the quasi-equivalence $\mathbf{p}'_A$, we refer to Example~\ref{exm:res}. The other two quasi-equivalences are induced by quasi-isomorphisms between dg algebras; see Example~\ref{exm:quasi}. The middle four dg categories are quasi-small. So we have to apply Remark~\ref{rem:quasi-sm} and view them in $\mathbf{Hodgcat}$.

By chasing the diagram, we observe that $H^0(\tilde{F})(T)$ is isomorphic to $A$. Consider the following composition
$$F_1\colon \mathbf{D}^b(A\mbox{-{\rm mod}}) \xrightarrow{F} \mathbf{D}^b(\mathcal{A}) \xrightarrow{{\rm can}_{A\mbox{-}{\rm mod}} H^0(\tilde{F}) {\rm can}_\mathcal{A}} \mathbf{D}^b(A\mbox{-{\rm mod}}).$$
We have an isomorphism $\theta\colon F_1(A)\rightarrow A$. Recall that  any morphism $a\colon A\rightarrow A$ in $\mathbf{D}^b(A\mbox{-mod})$ is given by the right multiplication by some element in $A$. Then by chasing the diagram for $\tilde{F}$, we observe that $\theta \circ F_1(a)=a\circ \theta$. By Lemma \ref{lem:pseu}, $F_1$ is isomorphic to a pseudo-identity. In particular, $F_1$ is an autoequivalence. Therefore, $H^0(\tilde{F})$ is also an equivalence, and thus by Lemma \ref{lem:H0} $\tilde{F}$ is an isomorphism in $\mathbf{Hodgcat}$. We observe that $F_2=F(F_1)^{-1}$ is liftable, whose dg lift is given by $(\tilde{F})^{-1}$.

For the uniqueness of factorizations, we just observe that $F'_1(F_1)^{-1}\simeq (F'_2)^{-1} F_2$ is liftable. By Theorem \ref{thm:lift}, $F'_1{F_1}^{-1}$ is standard, which is a pseudo-identity. By \cite[Lemma 5.9]{CY}, we infer  that $F'_1{F_1}^{-1}$ is isomorphic to the identity functor. Then we are done.
\end{proof}

\begin{cor}\label{cor:lift}
Assume that there are triangle equivalences among $\mathbf{D}^b(A\mbox{-{\rm mod}})$, $\mathbf{D}^b(\mathcal{A})$ and $\mathbf{D}^b(\mathcal{B})$. Then the following statements are equivalent:
\begin{enumerate}
\item The category $A\mbox{-}{\rm mod}$ is $\mathbf{D}$-standard;
\item Any triangle equivalence $\mathbf{D}^b(\mathcal{B}) \rightarrow \mathbf{D}^b(\mathcal{A})$ is liftable;
\item Any triangle autoequivalence on $\mathbf{D}^b(\mathcal{A})$ is liftable.
\end{enumerate}
\end{cor}

\begin{proof}
For ``(1) $\Rightarrow$ (2)", we apply Theorem \ref{thm:fac} to infer that all derived equivalences  $\mathbf{D}^b(A\mbox{-{\rm mod}})\rightarrow \mathbf{D}^b(\mathcal{A})$ and $\mathbf{D}^b(A\mbox{-{\rm mod}})\rightarrow \mathbf{D}^b(\mathcal{B})$ are liftable. Then (2) follows.  The implication ``(2) $\Rightarrow$ (3)" is trivial.

For ``(3) $\Rightarrow$ (1)", we take a pseudo-identity $F_1$ on $\mathbf{D}^b(A\mbox{-{\rm mod}})$. By Theorem \ref{thm:fac}, there is a liftable  equivalence $F_2\colon \mathbf{D}^b(A\mbox{-{\rm mod}})\rightarrow \mathbf{D}^b(\mathcal{A})$. Then $F_2F_1(F_2)^{-1}$ is a triangle autoequivalence on $\mathbf{D}^b(\mathcal{A})$, which is necessarily liftable. It follows that $F_1$ is also liftable. Recall that a liftable pseudo-identity is isomorphic to the identity functor; compare the last paragraph in the proof of Theorem \ref{thm:fac}. This proves (1).
\end{proof}

\begin{rem}
Keep the assumptions as above. We do not know the relation between these equivalent statements and the $\mathbf{D}$-standardness of the abelian categories $\mathcal{A}$ and $\mathcal{B}$.
\end{rem}

We are in a position to give the first proof to the theorem in the introduction. For a noetherian scheme $\mathbb{X}$, we denote by ${\rm coh}\mbox{-}\mathbb{X}$ the abelian category of coherent sheaves on $\mathbb{X}$.

\begin{thm}
Assume that there is a triangle equivalence between $\mathbf{D}^b(A\mbox{-}{\rm mod})$ and $\mathbf{D}^b({\rm coh}\mbox{-}\mathbb{X})$ with $\mathbb{X}$ a smooth projective scheme. Then $A\mbox{-}{\rm mod}$ is $\mathbf{D}$-standard, or equivalently, any triangle equivalence $F\colon \mathbf{D}^b(A\mbox{-}{\rm mod})\rightarrow \mathbf{D}^b(B\mbox{-}{\rm mod})$ is standard.
\end{thm}

\begin{proof}
Recall from \cite{Or} that any triangle autoequivalence on $\mathbf{D}^b({\rm coh}\mbox{-}\mathbb{X})$  is a Fourier-Mukai functor, and thus liftable by \cite[Proposition 6.11]{CS}. Applying Corollary \ref{cor:lift} and \cite[Theorem 5.10]{CY}, we are done.
\end{proof}

\section{The objective categories}

In this section, we introduce the notions of objective categories and triangle-objective triangulated categories.
The basic examples of triangle-objective triangulated categories are the bounded derived categories of  coherent sheaves over  projective varieties over an algebraically closed field.

We say that an endofunctor $F$ on a category $\mathcal{A}$ is \emph{object-preserving}, if $F(X)\simeq X$ for each object $X\in \mathcal{A}$.

\begin{defn}
An additive category $\mathcal{A}$ is called \emph{objective}, provided that any object-preserving autoequivalence on $\mathcal{A}$ is isomorphic to the identity functor ${\rm Id}_\mathcal{A}$.

Similarly, a triangulated category $\mathcal{T}$ is called \emph{triangle-objective}, provided that any object-preserving triangle autoequivalence on $\mathcal{T}$ is isomorphic, as a triangle functor,   to the identity functor ${\rm Id}_\mathcal{T}$.
\end{defn}

The following observation motivates the above notions.

\begin{lem}\label{lem:obj}
Let $\mathcal{A}$ be an abelian category. Consider the following statements:
\begin{enumerate}
\item The abelian  category $\mathcal{A}$ is $\mathbf{D}$-standard and objective;
\item  The bounded derived category $\mathbf{D}^b(\mathcal{A})$ is triangle-objective;
\item  The abelian  category $\mathcal{A}$ is $\mathbf{D}$-standard.
\end{enumerate}
Then we have the implications ``$(1) \Rightarrow (2) \Rightarrow (3)$".
\end{lem}

\begin{proof}
To see ``$(1) \Rightarrow (2)$", we take an object-preserving triangle autoequivalence $F$ on $\mathbf{D}^b(\mathcal{A})$. The restriction $F|_\mathcal{A}$ is object-preserving. By the assumptions in (1), we infer that $F|_\mathcal{A}$ is isomorphic to the identity functor ${\rm Id}_\mathcal{A}$. By \cite[Corollary 3.9]{CY}, $F$ is isomorphic to a pseudo-identity on $\mathbf{D}^b(\mathcal{A})$. Since $\mathcal{A}$ is $\mathbf{D}$-standard, we infer that $F$ is isomorphic to the identity functor.  The implication ``$(2)\Rightarrow (3)$" is clear, since any pseudo-identity is object-preserving.
\end{proof}

Let $R$ be a commutative noetherian $k$-algebra.  Denote by $R\mbox{-mod}$ the abelian category of finitely generated $R$-modules.

Given a $k$-algebra automorphism $\sigma\colon R\rightarrow R$ and an $R$-module $M$, we denote by ${^\sigma}(M)$ the twisted module: the new $R$-action is given by $a_\circ m=\sigma^{-1}(a).m$, where the dot ``$.$" denotes the $R$-action on $M$. This gives rise to the \emph{twist automorphism}
$$^\sigma(-)\colon R\mbox{-mod}\longrightarrow R\mbox{-mod}.$$

\begin{exm}
 {\rm Denote by $k[\epsilon]$ the algebra of dual numbers. By \cite[Theorem~7.1]{CY}, $k[\epsilon]\mbox{-mod}$ is $\mathbf{D}$-standard. However, $k[\epsilon]\mbox{-mod}$ is not objective and $\mathbf{D}^b(k[\epsilon]\mbox{-mod})$ is not triangle-objective, provided that the field $k$ contains at least three elements.

 Fix $a\in k$ satisfying $a\neq 0, 1$. Consider the automorphism $\sigma$ on $k[\epsilon]$ such that $\sigma(\epsilon)=a\epsilon$. The twist automorphisms $^\sigma(-)$,  defined on $k[\epsilon]\mbox{-mod}$ and  $\mathbf{D}^b(k[\epsilon]\mbox{-mod})$,  are both object-preserving, but neither is isomorphic to the identity functor.}
\end{exm}

The following condition arises naturally.

\vskip 3pt

 Condition $({\rm Obj})$:  any $k$-algebra automorphism $\sigma\colon R\rightarrow R$ satisfying $\sigma(I)=I$ for each ideal $I$, necessarily equals ${\rm Id}_R$.
\vskip 3pt

\begin{lem}
Let $R$ be a commutative noetherian ring satisfying {\rm Condition $({\rm Obj})$}. Then $R\mbox{-}{\rm mod}$ is objective.
\end{lem}

\begin{proof}
Assume that   $F\colon R\mbox{-mod}\rightarrow R\mbox{-mod}$ is  an object-preserving autoequivalence. Since $F(R)\simeq R$, it follows that $F$ is isomorphic to the twist automorphism ${^\sigma}(-)$ for some automorphism $\sigma$. We observe that ${^\sigma}(R/I)\simeq R/\sigma(I)$ for each ideal $I$. By the isomorphism ${^\sigma}(R/I)\simeq R/I$ and taking their annihilator ideals, we infer that $\sigma(I)=I$. By Condition ${\rm (Obj)}$, we have $\sigma={\rm Id}_R$. Consequently, $F$ is isomorphic to the identity functor.
\end{proof}

Here are some examples of rings satisfying Condition  $({\rm Obj})$.

\begin{exm}\label{exm:obj}
{\rm (1) The polynomial algebras satisfy Condition  $({\rm Obj})$. More generally, we assume that $R$ is an integral domain such that any invertible element is a scalar. Then $R$ satisfies Condition  $({\rm Obj})$.

To verify the condition, we take an automorphism  $\sigma\colon R\rightarrow R$ satisfying $\sigma(I)=I$. For any non-scalar $a\in R$, we have $\sigma(Ra)=R\sigma(a)=Ra$. It follows that $\sigma(a)=\lambda a$ for some $\lambda\in k$. Similarly, $\sigma(1+a)=\lambda'(1+a)$ for some $\lambda'\in k$. By comparing these two identities, we infer that $\lambda=1=\lambda'$.

(2) Any reduced affine algebra over an algebraically closed field satisfies Condition  $({\rm Obj})$. More generally, we assume that the Jacobson radical of $R$ is zero and that for each maximal ideal $\mathfrak{m}$, the natural homomorphism $k\rightarrow R/\mathfrak{m}$ is an isomorphism.  Then $R$ satisfies Condition  $({\rm Obj})$.

For the verification, we claim that $a-\sigma(a)$ is contained in any maximal ideal $\mathfrak{m}$. By assumption, there is some $\lambda\in k$ satisfying $a-\lambda\in \mathfrak{m}$. Then we have  $\sigma(a)-\lambda\in \sigma(\mathfrak{m})=\mathfrak{m}$. The claim follows immediately.}
\end{exm}

The following result shows that objective categories are ubiquitous in algebraic geometry. For a sheaf $\mathcal{F}$, we denote by ${\rm supp}(\mathcal{F})$ its support.

\begin{prop}\label{prop:coh}
Let $(\mathbb{X}, \mathcal{O})$ be a noetherian scheme such that there is a finite affine open covering $\mathbb{X}=\bigcup U_i$, where $U_i={\rm Spec}(R_i)$ with each $R_i$ satisfying {\rm Condition} $({\rm Obj})$. Then ${\rm coh}\mbox{-}\mathbb{X}$ is objective.

Assume further that $\mathbb{X}$ is projective such that the maximal torsion subsheaf $T_0(\mathcal{O})\subseteq \mathcal{O}$ of dimension zero is trivial. Then $\mathbf{D}^b({\rm coh}\mbox{-}\mathbb{X})$ is triangle-objective.
\end{prop}

\begin{proof}
Let $F\colon {\rm coh}\mbox{-}\mathbb{X}\rightarrow {\rm coh}\mbox{-}\mathbb{X}$ be an object-preserving auto-equivalence. In particular, $F$ fixes the structure sheaf $\mathcal{O}$.  It is well-known that there is a unique automorphism $\theta$ on $\mathbb{X}$ such that $F\simeq \theta^*$, the pullback functor; see \cite[Theorem 5.4]{Bra}.

For each closed subset $Z\subseteq \mathbb{X}$, we have an ideal sheaf $\mathcal{I}$ with ${\rm supp}(\mathcal{O}/\mathcal{I})=Z$. Then we have ${\rm supp}(\theta^*(\mathcal{O}/\mathcal{I}))=\theta^{-1}(Z)$. By the isomorphism $\theta^*(\mathcal{O}/\mathcal{I})\simeq \mathcal{O}/\mathcal{I}$, we infer that $\theta^{-1}(Z)=Z$.  In particular, for the given affine open subsets $U_i$, we have $\theta^{-1}(U_i)=U_i$.  Therefore, the restriction $\theta|_{U_i}\colon U_i\rightarrow U_i$ corresponds to an $k$-algebra automorphism $\sigma_i$ on $R_i$, that is, $\theta|_{U_i}={\rm Spec}(\sigma_i)$.

We have the following commutative diagram
\[\xymatrix{
{\rm coh}\mbox{-}\mathbb{X}\ar[d]^-{\theta^*} \ar[r]^{\rm res} & {\rm coh}\mbox{-}U_i \ar[d]^-{(\theta|_{U_i})^*} \ar@{=}[r] & R_i\mbox{-mod} \ar[d]^-{^{\sigma_i}(-)}\\
{\rm coh}\mbox{-}\mathbb{X} \ar[r]^{\rm res} & {\rm coh}\mbox{-}U_i \ar@{=}[r] & R_i\mbox{-mod},
}\]
where ``res" is the restriction functor, and we identify ${\rm coh}\mbox{-}U_i$ with $R_i\mbox{-mod}$. The restriction functor ``res" induces the well-known equivalence between ${\rm coh}\mbox{-}U_i$ and the Serre quotient category of ${\rm coh}\mbox{-}\mathbb{X}$ by those sheaves supported on the complement of $U_i$; compare \cite[Example 4.3]{Bra}. It follows that $(\theta|_{U_i})^*$ and thus ${^{\sigma_i}(-)}$ are object-preserving. By the assumption on $R_i$, it follows that $\sigma_i={\rm Id}_{R_i}$ and thus $\theta|_{U_i}={\rm Id}_{U_i}$ for each $i$. Therefore, $\theta={\rm Id}_\mathbb{X}$, proving the first statement.

For the last statement, we apply \cite[Lemma 9.2]{LO} to infer that ${\rm coh}\mbox{-}\mathbb{X}$ has an ample sequence in the sense of \cite{Or}. By \cite[Proposition 5.7]{CY}, we deduce that ${\rm coh}\mbox{-}\mathbb{X}$ is $\mathbf{D}$-standard. Using the proved statement and Lemma \ref{lem:obj}, we are done.
\end{proof}

By Example \ref{exm:obj}(2),  a reduced projective scheme over an algebraically closed field satisfies the above conditions. Hence, the following immediate consequence of Proposition \ref{prop:coh} and Lemma \ref{lem:obj} gives the second proof to the theorem in the introduction, when the field $k$ is algebraically closed. As a consequence, the smooth hypothesis of the scheme can be relaxed.

\begin{cor}
Let $A$ be a finite dimensional algebra. Assume that there is a triangle equivalence between $\mathbf{D}^b(A\mbox{-}{\rm mod})$ and  $\mathbf{D}^b({\rm coh}\mbox{-}\mathbb{X})$ for a projective scheme $\mathbb{X}$ satisfying the conditions in Proposition \ref{prop:coh}. Then $\mathbf{D}^b(A\mbox{-}{\rm mod})$ is triangle-objective, and thus  $A\mbox{-}{\rm mod}$ is $\mathbf{D}$-standard. \hfill $\square$
\end{cor}

\vskip 5pt

\noindent {\bf Acknowledgements}.\quad  This work is supported by the National Natural Science Foundation of China (Nos. 11522113 and 11671245),  the Fundamental Research Funds for the Central Universities,  and Anhui Initiative in Quantum Information Technologies (AHY150200).

\bibliography{}

\begin{thebibliography}{9999}

\bibitem{Bra} {\sc M. Brandenburg}, {\em Rosenberg's reconstruction theorem}, Expo. Math. {\bf 36} (2018), 98--117.


\bibitem{CS15} {\sc A. Canonaco, and P. Stellari}, {\em Internal hom via extensions of dg functors}, Adv. Math. {\bf 277} (2015), 100--123.

\bibitem{CS} {\sc A. Canonaco, and P. Stellari}, {\em A tour about existence and uniqueness of dg enhancements and lifts}, J. Geom. Phys. {\bf 122} (2017), 28--52.

\bibitem{C} {\sc X.W. Chen}, {\em Representablity and autoequivalence groups}, arXiv:1810.00332v2, 2018.

\bibitem{CY} {\sc X.W. Chen, and Y. Ye}, {\em The $\mathbf{D}$-standard and $\mathbf{K}$-standard categories}, Adv. Math. {\bf 333} (2018), 159--193.

    \bibitem{CZ} {\sc X.W. Chen, and C. Zhang}, {\em The derived-discrete algebras and standard equivalences}, J. Algebra {\bf 525} (2019), 259--283. 

\bibitem{Dri} {\sc V. Drinfeld}, {\em DG quotients of DG categories}, J. Algebra {\bf 272} (2004), 643--691.


\bibitem{Kel94} {\sc B. Keller}, {\em Deriving DG categories}, Ann. Sci. \'{E}cole Norm Sup. (4) {\bf 27}(1) (1994), 63--102.

\bibitem{Kel05} {\sc B. Keller}, {\em On triangulated orbit categories}, Doc. Math. {\bf 10} (2005), 551--581.


\bibitem{Kr} {\sc H. Krause}, {\em Completing perfect complexes}, arXiv:1805.10751v3, 2018.


\bibitem{LO} {\sc V. Lunts, and D. Orlov}, {\em Uniqueness of enhancements for triangulated categories}, J. Amer. Math. Soc. {\bf 23} (2010), 853--908.


\bibitem{Or} {\sc D. Orlov}, {\em Equivalences of derived categories and K3 surfaces}, J. Math. Sci. {\bf 84} (1997), 1361--1381.

\bibitem{Ric} {\sc J. Rickard}, {\em Derived equivalences as derived functors}, J. London Math. Soc. {\bf 43}(2) (1991), 37--48.

\bibitem{RVan} {\sc A. Rizzardo, and M. Van den Bergh}, {\em An example of a non-Fourier-Mukai functor between derived categories of coherent sheaves}, arXiv:1410.4039v2, 2015.


\bibitem{Toe1} {\sc B. To\"{e}n}, {\em The homotopy category of dg-categories and derived Morita theory}, Invent. Math. {\bf 167} (2007), 615--667.

\bibitem{Toe2} {\sc B. To\"{e}n}, {\em Lectures on dg-categories}, in: Topics in Algebraic and Topological K-Theory, Lect. Notes Math. {\bf 2008}, 243--302, Springer Berlin Heidelberg, 2011.

\end{thebibliography}

\vskip 10pt

 {\footnotesize \noindent Xiaofa Chen, Xiao-Wu Chen\\
 Key Laboratory of Wu Wen-Tsun Mathematics, Chinese Academy of Sciences,\\
 School of Mathematical Sciences, University of Science and Technology of China, Hefei 230026, Anhui, PR China}

\end{document}